\documentclass[11pt]{amsart}
\usepackage{amssymb,amstext,amsmath,amscd,amsthm,amsfonts,graphicx,latexsym}
\usepackage[usenames]{color}
\usepackage[all]{xy}
\usepackage{array,
colonequals,enumerate,float,gensymb,mathrsfs,mathtools,
}
\usepackage[usenames]{color}
\newtheorem{thm}{Theorem}[section]

\newtheorem{cor}[thm]{Corollary}
\newtheorem{lem}[thm]{Lemma}
\newtheorem{prop}[thm]{Proposition}
\theoremstyle{definition}
\newtheorem{dfn}[thm]{Definition}
\newtheorem{rem}[thm]{Remark}
\newtheorem{ques}[thm]{Question}

\newtheorem{ex}[thm]{Example}

\newtheorem{claim}{Claim}
\newtheorem*{claim*}{Claim}
\theoremstyle{remark}

\newtheorem*{ac}{Acknowlegments}

\numberwithin{equation}{thm}

\def\Hom{\operatorname{Hom}}
\def\Ext{\operatorname{Ext}}

\def\ann{\operatorname{ann}}

\def\depth{\operatorname{depth}}

\def\End{\operatorname{End}}

\def\ker{\operatorname{Ker}}
\def\coker{\operatorname{Coker}}
\def\image{\operatorname{Im}}
\def\p{\mathfrak{p}}

\def\ge{\geqslant}
\def\geq{\geqslant}
\def\le{\leqslant}
\def\leq{\leqslant}

\def\iff{\Leftrightarrow}
\def\implies{\Rightarrow}

\def\edim{\operatorname{edim}}

\def\m{\mathfrak{m}}
\def\n{\mathfrak{n}}
\begin{document}
\allowdisplaybreaks
\title[Rings with self-dual maximal ideals]{Local rings with self-dual maximal ideal}
\author{Toshinori Kobayashi}
\address{Graduate School of Mathematics, Nagoya University, Furocho, Chikusaku, Nagoya, Aichi 464-8602, Japan}
\email{m16021z@math.nagoya-u.ac.jp}
\thanks{The auther was supported by JSPS Grant-in-Aid for JSPS Fellows 18J20660.}
\begin{abstract}
	Let $R$ be a Cohen-Macaulay local ring possessing a canonical module.
	In this paper we consider when the maximal ideal of $R$ is self-dual, i.e.\ it is isomorphic to its canonical dual as an $R$-module.
	local rings satisfying this condition are called Teter rings, and studied by Teter, Huneke-Vraciu, Ananthnarayan-Avramov-Moore, and so on.
	In the one dimensional case, we show such rings are exactly the endomorphism rings of the maximal ideals of some Gorenstein local rings of dimension one.
	We also provide some connection between the self-duality of the maximal ideal and near Gorensteinness.
\end{abstract}
\maketitle
\section{Introduction}

Let $R$ be a Cohen-Macaulay local ring with a canonical module $\omega$.
For an $R$-module $M$, we denote by $M^\dag$ the $R$-module $\Hom_R(M,\omega)$.
The $R$-module $M$ is called {\it self-dual} if there exists an isomorphism $M\xrightarrow[]{\cong} M^\dag$ of $R$-modules.
Note that the self-duality of $R$-modules is independent of the choice of $\omega$.

Let $R$ and $S$ be artinian local rings such that $S$ maps onto $R$.
Denote by $c_S(R)$ the colength $\ell_S(S)-\ell_S(R)$.
In the case that $S$ is Gorenstein, the integer $c_S(R)$ is used to estimate  homological properties of $R$, for example, see \cite[Theorem 7.5]{Kustin-Vraciu}.
Ananthnarayan \cite{Ananthnarayan} introduced the {\it Gorenstein colength} $g(R)$ of an artinian local ring $(R,\m,k)$ to be the following integer
\[
g(R)\coloneqq\min\{c_S(R) \mid S \text{ is a Gorenstein artinian local ring mapping onto } R\}.
\]

The number $g(R)$ measures how close $R$ is to a Gorenstein ring.
Clearly, $g(R)$ is zero if and only if $R$ is Gorenstein.
One can see that $g(R)=1$ if and only if $R$ is non-Gorenstein and $R\cong S/\mathrm{soc}(S)$ for an artinian Gorenstein ring $S$.
These rings are called {\em Teter rings}.
On Teter rings, the following characterization is known, which is an improvement of Teter's result \cite{Teter}.
This was proved by Huneke-Vraciu \cite{Huneke-Vraciu} under the assumption that $1/2\in R$ and $\mathrm{soc}(R)\subseteq\m^2$, and later Ananthnarayan-Avramov-Moore \cite{A-A-M} removed the assumption $\mathrm{soc}(R)\subseteq\m^2$.
See also the result of Elias-Takatsuji \cite{{Elias-Takatuji}}.


\begin{thm} [Huneke-Vraciu, Ananthnarayan-Avramov-Moore, Elias-Takatuji] \label{thHVAAM}
Let $(R,\m,k)$ be an artinian local ring such that either $R$ contains $1/2$ or $R$ is equicharacteristic with $\mathrm{soc}(R)\subseteq \m^2$.
Then the following are equivalent.
\begin{itemize}
\item[(1)] $g(R)\leq 1$.
\item[(2)] Either $R$ is Gorenstein or $\m\cong \m^\dag$.
\item[(3)] Either $R$ is Gorenstein or there exists a surjective homomorphism $\omega \to \m$.
\end{itemize}
\end{thm}

Moreover, Ananthnarayan \cite{Ananthnarayan} extended this theorem to the case $g(R)\leq 2$ as follows.

\begin{thm} [Ananthnarayan] \label{Anan}
Let $(R,\m)$ be an artinian local ring. Write $R\cong T/I$ where $(T,\m_T)$ is a regular local ring and $I$ is an ideal of $T$.
Suppose $I\subseteq \m_T^6$ and $1/2\in R$.
Then the following are equivalent.
\begin{itemize}
\item[(1)] $g(R)\leq 2$.
\item[(2)] There exists a self-dual ideal $\mathfrak{a}\subseteq R$ such that $l(R/\mathfrak{a})\leq 2$.
\end{itemize}
\end{thm}

In this paper, we try to extend the notion of Gorenstein colengths and the above results to the case that $R$ is a one-dimensional Cohen-Macaulay local ring.

For a local ring $(R,\m)$, we denote by $Q(R)$ the total quotient ring of $R$. An extension $S\subseteq R$ of local rings is called {\it birational} if $R\subseteq Q(S)$.
In this case, $R$ and $S$ have same total quotient ring.  

Let $(S,\mathfrak{n})\subseteq (R,\m)$ be an extension of local rings.
Suppose $\mathfrak{n}=\m\cap S$.
Then $S\subseteq R$ is called {\it residually rational} if there is an isomorphism $S/\mathfrak{n}\cong R/\m$ induced by the natural inclusion $S\to R$.
For example, if $S\subseteq R$ is module-finite and $S/\mathfrak{n}$ is algebraically closed, then it automatically follows that $S\subseteq R$ is residually rational.
We introduce an invariant $bg(R)$ for local rings $R$ as follows, which is the infimum of Gorenstein colengths in birational maps.

\begin{dfn}
For a local ring $R$, we define
\[
bg(R)\coloneqq\inf\left\{\ell_S(R/S) \middle|
\begin{array}{l}
 S\text{ is Gorenstein and }S\subseteq R\text{ is a module-finite}\\
\text{residually rational birational map of local rings}
\end{array}
\right\}.
\]
\end{dfn}

We will state the main results of this paper by using this invariant.
The first one is the following theorem, which gives a one-dimensional analogue of Theorem \ref{thHVAAM}.

\begin{thm} \label{thA}
Let $(R,\m)$ be a one-dimensional Cohen-Macaulay local ring having a canonical module $\omega$.
Consider the following conditions.
\begin{itemize}
\item[(1)] $bg(R)\leq 1$.
\item[(2)] Either $R$ is Gorenstein or there exists a Gorenstein local ring $(S,\mathfrak{n})$ of dimension one such that $R\cong \End_S(\mathfrak{n})$.
\item[(3)] Either $R$ is Gorenstein or  $\m\cong \m^\dag$.
\item[(4)] Either $R$ is Gorenstein or there is a short exact sequence $0\to \omega \to \m \to k \to 0$.
\item[(5)] There is an ideal $I$ of $R$ such that $I\cong \omega$ (i.e. $I$ is a canonical ideal of $R$) and $l(R/I)\leq 2$.
\end{itemize}
Then the implications $(1)\implies (2) \implies (3) \iff (4) \iff (5)$ hold. The direction $(5) \implies (1)$ also holds  if $R$ contains an infinite field $k$ as a subalgebra which is isomorphic to $R/\m$ via the projection $R\to R/\m$, i.e. $R$ has an infinite coefficient field $k\subseteq R$.
\end{thm}

The existence of a canonical ideal $I$ of $R$ with $\ell_R(R/I)=2$ is considered by Dibaei-Rahimi \cite{DR}.
Using their notion, the condition (5) above is equivalent to the condition that $\min(S_{\mathfrak{C}_R})\le 2$.

We also remark that Bass's idea \cite{B} tells us the importance of the endomorpshism ring $\End_S(\mathfrak{n})$ of the maximal ideal $\mathfrak{n}$ of a Gorenstein local ring $S$ of dimension one.
He shew that any torsion-free $S$-module without non-zero free summand can be regarded as a module over $\End_S(\mathfrak{n})$.
So we can analyze Cohen-Macaulay representations of $R$ via the ring $\End_S(\mathfrak{n})$ (see also \cite[Chapter 4]{LW}).

As a corollary, we can characterize Cohen-Macaulay local rings $R$ of dimension one having minimal multiplicity and satisfying $bg(R)\leq 1$.
To give the statement of our corollary, we recall some definitions.
For a local ring $R$, we denote by $e(R)$ the multiplicity of $R$, $r(R)$ the Cohen--Macaulay type, and by $\edim R$ the embedding dimension of $R$.
According to Goto-Matsuoka-Phuong \cite{GMP}, a local ring $R$ is called {\em almost Gorenstein}, if $R$ possesses a canonical ideal $I$ of $R$ such that $e_1(I)\le r(R)$, where $e_1(I)$ is the first Hilbert coefficient of $I$.
A Gorenstein ring of dimension one satisfying $e(S)= \edim S+1$ are called {\em a ring of almost minimal multiplicity} or a {\em Gorenstein ring of minimal multiplicity}, and studied by J.\ D.\ Sally \cite{Sally}.
The invariant $\rho(R)$ is the {\em canonical index} of $R$, introduced by Ghezzi-Goto-Hong-Vasconcelos \cite{GGHV}.

\begin{cor} \label{co1}
Let $(R,\m)$ be a one-dimensional Cohen-Macaulay local ring.
Consider the following conditions.
\begin{itemize}
\item[(1)] $bg(R)\leq 1$ and $R$ has minimal multiplicity.
\item[(2)] Either $e(R)\leq 2$ or $R$ is almost Gorenstein with $bg(R)=1$.
\item[(3)] $\m\cong \m^\dag$ and $R$ is almost Gorenstein.
\item[(4)] $\m\cong \m^\dag$ and $R$ has minimal multiplicity.
\item[(5)] $R$ is almost Gorenstein and has minimal multiplicity.
\item[(6)] There exists a Gorenstein local ring $(S,\mathfrak{n})$ of dimension one such that $e(S)\leq \edim S+1$ and $R\cong \End_S(\mathfrak{n})$.
\item[(7)] $\m\cong \m^\dag$ and $\rho(R)\le 2$.
\end{itemize}
Then $(1) \iff (2) \implies (3) \iff (4) \iff (5)$ holds. If $R/\m$ is infinite, then $(5) \iff (7)$ and $(6) \implies (4)$ hold.
If $R$ has an infinite coefficient field $k\subseteq R$, then all the conditions are equivalent.
\end{cor}

The second main theorem of this paper is the following, which is a one-dimensional analogue of Theorem \ref{Anan}.

\begin{thm} \label{thB}
Let $(R,\m)$ be a complete one-dimensional Cohen-Macaulay local ring. Consider the following conditions.
\begin{itemize}
\item[(1)] $bg(R)\leq 2$.
\item[(2)] There exists a self-dual ideal $\mathfrak{a}\subseteq R$ such that $\ell_R(R/\mathfrak{a})\leq 2$.
\end{itemize}
Then $(1)$ implies $(2)$. The implication $(2) \implies (1)$ also holds if $R$ has an infinite coefficient field $k\subseteq R$.
\end{thm}

In the view of Theorem \ref{thA}, local rings with self-dual maximal ideal are naturally constructed from Gorenstein local rings, and so their ubiquity is certified.
It is interesting to consider what good properties they have compared to Gorenstein rings.
In section 3, we have an observation that a Cohen-Macaulay local ring $(R,\m)$ is nearly Gorenstein (see Definition \ref{deftr} for the definition) if $\m$ is self-dual.
The converse of this is not true in general, however, we have the following result.
Here $\m:\m$ is a subring of $Q(R)$ consisting of the elements $a$ satisfying $a\m\subseteq \m$.

\begin{thm} \label{th313}
Let $(R,\m,k)$ be a Cohen-Macaulay local ring of dimension one.
Put $B=\m:\m$.
Assume $k$ is infinite.
\begin{itemize}
\item[(1)] If $B$ is local with Cohen-Macaulay type two and $R$ is nearly Gorenstein, then $R$ is almost Gorenstein and does not satisfy $\m\cong \m^\dag$.
\item[(2)] If $B$ is local with Cohen-Macaulay type three and $R$ is nearly Gorenstein, then either $R$ is almost Gorenstein or $\m\cong \m^\dag$.
\end{itemize}
\end{thm}

We will provide a proof of Theorem \ref{th313} in section 3.
One should compare this theorem with the following result of Goto-Matsuoka-Phuong \cite[Theorem 5.1]{GMP}.

\begin{thm}[Goto-Matsuoka-Phuong] \label{LSth18}
Let $(R,\m,k)$ be a Cohen-Macaulay local ring of dimension one.
Put $B=\m:\m$.
Then $B$ is Gorenstein if and only if $R$ is almost Gorenstein and has minimal multiplicity.
\end{thm}

In section 4, we deal with numerical semigroup rings having self-dual maximal ideal.
The definition of UESY-semigroups was given by \cite{Rosales}.
These numerical semigroups are exactly the semigroups obtained by adding one element to a symmetric numerical semigroup.
We will show that a numerical semigroup ring has self-dual maximal ideal if and only if the corresponding numerical semigroup is UESY.
After that, we also prove that the rings of UESY-numerical semigroup have quasi-decomposable maximal ideal.
According to \cite{NT}, an ideal $I$ of $R$ is called {\it quasi-decomposable} if there exists a regular sequence $\underline{x}=x_1,\dots,x_t$ such that $I/(\underline{x})$ is decomposable as an $R$-module.
Local rings with quasi-decomposable maximal ideal have some interesting properties; we can classify thick subcategories of the singularity category with some assumption on the punctured spectrum (\cite[Theorem 4.5]{NT}), and we have results on the vanishings of Ext and Tor (\cite[Section 6]{NT}).

In section 5, we characterize the endomorphism ring of a local hypersurface of dimension one, using Theorem \ref{thA}.

\section{Proof of Theorem \ref{thA} and \ref{thB}}
In this section, we prove Theorem \ref{thA} and \ref{thB}.
Let $(R,\m)$ be a Noetherian local ring with total quotient ring $Q(R)$.
Denote by $\widetilde R$ the integral closure of $R$ in $Q(R)$.
A {\it fractional ideal} is a finitely generated $R$-submodule $I$ of $Q(R)$ with $IQ(R)=Q(R)$ (i.e. $I$ contains a non-zerodivisor of $R$).  If $R$ has $\depth R \ge 1$, then every $\m$-primary ideal is a fractional ideal of $R$.
For fractional ideals $I$ and $J$, $J:I$ means the $R$-submodule $\{a\in Q(R)\mid aI\subseteq J\}$ of $Q(R)$.
Then the assignments $\alpha \mapsto [x\mapsto \alpha x]$ make an isomorphism $\Phi_{I,J}:J:I\cong\Hom_R(I,J)$ of $R$-modules; see \cite[Proposition 2.1(1)]{KT} for instance.
In this way, the endomorphism ring $\End_R(\m)$ of $\m$ is identified with the $R$-subalgebra $\m:\m$ of $Q(R)$.
As the $R$-module $\Hom_R(\m,\m)$ is finitely generated ,the extension $R\subseteq \m:\m$ is module-finite.
Hence $\m:\m$ is semilocal.


We give the following lemma in order to use in the proof of Theorem \ref{thA}.

\begin{lem} \label{lem21}
Let $(R,\m)$ be a Cohen-Macaulay local ring of dimension one.
Assume $R$ is not a  discrete valuation ring.
Then 
\begin{itemize}
\item[(a)]
$\m:\m=R:\m$
\item[(b)]
$\ell_R(\m:\m/R)$ is equal to $r(R)$.
In particular, $R$ is properly contained in $\m:\m$.
\end{itemize}
\end{lem}

\begin{proof}
(a), (b) The inclusion $\m\subseteq R$ induces a natural inclusion $\m:\m\subseteq R:\m$.
Suppose this inclusion is strict.
Then there is an element $a\in R:\m$ which does not satisfy $a\m\subseteq \m$.
Since $R$ is local and $a\m$ is an ideal of $R$, this means that the equality $a\m=R$ holds.
Multiplying by the denominator of $a$ on both sides of the equality,
we have an equality $b\m =(c)$, where $b$ is in $R$ and $c$ is a non-zerodivisor of $R$.
By this equality, we may see that $\ann_R(b)\subseteq \ann_R(c)=0$.
It means that $b$ is also a non-zerodivisor of $R$.
Thus $b$ induces a homomorphism $\m \to (c)$ which is both injective and surjective.
In particular, $\m$ is isomorphic to $R$.
This yields that $R$ is a discrete valuation ring.
This contradicts our assumption.
Consequently the equality $\m:\m=R:\m$ holds.

Next we look at the long exact sequence
\[
0 \to \Hom_R(R/\m,R) \to \Hom_R(R,R) \xrightarrow[]{f} \Hom_R(\m,R) \to \Ext_R^1(R/\m,R) \to \Ext_R^1(R,R)=0
\]
induced by the short exact sequence $0 \to \m \to R \to R/\m \to 0$.

The length of the $R$-module $\Ext_R^1(R/\m,R)$ is exactly equal to the Cohen--Macaulay type $r(R)$.
On the other hand, one can directly check that the natural inclusion $g:R \to \m:\m=R:\m$ satisfies the equality $\Phi_{R,\m}\circ g=f\circ \Phi_{R,R}$, where $\Phi_{R,\m}\colon R:\m \to \Hom_R(\m,R)$ and $\Phi_{R,R}\colon R:R\to \Hom_R(R,R)$ are isomorphisms as in the initial paragraph of this section.
It follows that the cokernel $\m:\m/R$ of $g$ is isomorphic to $\Ext_R^1(R/\m,R)$, the cokernel of $f$.
In particular, both two $R$-module have same length.
We thus obtain the equality $\ell_R(\m:\m/R)=r(R)$.

\end{proof}

\begin{lem} \label{lem22}
Let $(S,\mathfrak{n}) \subsetneq (R,\m)$ be a module-finite birational extension of one-dimensional local rings.
Assume $R$ is reflexive as an $S$-module.
Then we have birational extensions $S\subsetneq \mathfrak{n}:\mathfrak{n} \subseteq R$.
\end{lem}

\begin{proof}
Note that $S$ is not a discrete valuation ring, and so $S$ is properly contained in $\mathfrak{n}:\mathfrak{n}$ (Lemma \ref{lem21}).
By the assumption, $R$ is reflexive, i.e. $R=S:(S:R)$.

We use the following claims.

\begin{claim} \label{cla1}
The $S$-modules $R$ and $S:R$ have no nonzero $S$-free summands.
\end{claim}

\begin{proof}[Proof of Claim 1]
First consider the case when $R$ has an $S$-free summand, that is, $R\cong S \oplus X$ for some $S$-module $X$.
Then $R\otimes_S Q(S)\cong Q(S)\oplus X\otimes_S Q(S)$.
Since $S\subseteq R$ is finite birational, $R\otimes_S Q(S)\cong Q(R)=Q(S)$.
Therefore, we obtain that $Q(S)$ is isomorphic to $Q(S)\oplus X\otimes_S Q(S)$.
Since $Q(S)$ is artinian, we may use the Krull--Schmidt theorem for $Q(S)$ to show that $X\otimes_S Q(S)=0$.
In particular, $X$ is a torsion $S$-module.
However, $X$ is a submodule of torsionfree $S$-module $R$, and hence $X$ itself is torsionfree.
Thus $X$ should be a zero module.
This shows that $R\cong S$.
As $R$ is a finite module over $S$, the ring-extension $R\subseteq S$ is integral.
Thus $R$ is contained in the integral closure of $S$ in $Q(S)$.
By \cite[Lemma 2.13]{Ko}, it follows that $R=S$.
This is a contradiction.

Now suppose that $S:R$ has an $S$-free summand.
Since $(S:R)\cong \Hom_S(R,S)$ is an $R$-module via the action of $R$ on the left-hand side in the homomorphism set, one has a surjective $R$-homomorphism $R^n \to (S:R)$, where $R^n$ is a nonzero finite free $R$-module.
Compositing this map with the split surjection $(S:R) \to S$, we have a surjective $S$-homomorphism $R^n \xrightarrow[]{[f_1,\dots,f_n]}S$.
Since $S$ is local, one of $f_1,\dots,f_n$ is surjective.
This implies that there is a surjection $R\to S$.
By the fact that $S$ is a projective $S$-module, it induces that $S$ is a direct summand of $R$.
Then by the previous argument, it forces that $R=S$, a contradiction.
\end{proof}

\begin{claim}
Let $M$ be an $S$-module without free summand.
Then $\Hom_S(M,S)$ is isomorphic to $\Hom_S(M,\n)$.
\end{claim}

\begin{proof}[Proof of Claim 2]
Since $S$ is local, for any homomorphism $f\colon M\to S$ the image of $f$ is contained in $\n$ (otherwise it would produce a non-trivial free summand of $M$).
Hence the natural injection $\Hom_S(M,\n) \to \Hom_S(M,S)$ is an isomorphism.
\end{proof}

By the two claims above, one has $R=\mathfrak{n}:(\mathfrak{n}:R)$.
So we may view $R$ as an $\n:\n$-module $\Hom_S(\Hom_S(R,\n),\n)$ via the action of $\n:\n$ on $\n$.
Therefore, we get equalities $R=(\n:\n)R$ and $(\mathfrak{n}:\mathfrak{n})\subseteq(\mathfrak{n}:\mathfrak{n})R\subseteq R$.
\end{proof}

Now we can explain the proof of the direction \ref{thA} (1) $\implies$ (2) $\implies$ (3) $\iff$ (4) $\iff$ (5) of Theorem \ref{thA}. 

\begin{proof}[Proof of Theorem \ref{thA} (1) $\implies$ (2) $\implies$ (3) $\iff$ (4) $\iff$ (5)]

(1) $\implies$ (2): Assume $bg(R)\leq 1$.
If $bg(R)=0$, then $R$ is Gorenstein, and there is nothing to prove.
We may assume $bg(R)=1$.
Then there is a Gorenstein local ring $(S,\mathfrak{n})$ and module-finite residually rational birational extension $S\subsetneq R$ with $\ell_S(R/S)=1$.
Since $S$ is Gorenstein and $R$ is maximal Cohen--Macaulay over $S$, $R$ is reflexive as an $S$-module (see \cite[Theorem 3.3.10]{BH} for instance).
By the previous lemma, we have $S\subsetneq \mathfrak{n}:\mathfrak{n}\subseteq R$.
Therefore, it should follows that $\ell_S(R/\mathfrak{n}:\mathfrak{n})=0$, in other words, $R=\mathfrak{n}:\mathfrak{n}=S:\mathfrak{n}$.

(2) $\Rightarrow$ (3): We may assume $S$ is not a discrete valuation ring (otherwise $R\cong S$ and hence $R$ is Gorenstein).
Identify $R$ with $\mathfrak{n}:\mathfrak{n}$.
Since we can check immediately that $\m\cong\m^\dag$ if $R$ is a discrete valuation ring.
So we may assume $R$ is not a discrete valuation ring.
By Lemma \ref{lem21}, one has $\ell_S(R/S)=1$.
Hence we have that the colength $\ell_S(\m/\mathfrak{n})$ of the inclusion $\mathfrak{n}\subseteq \m$ is less than or equal to $1$.
Remark that $\n$ cannot be equal to $\m$, otherwise we get an equality $R=\n:\n=\m:\m$, which forces that $R$ is a discrete valuation ring by Lemma \ref{lem21}.
Thus $\ell_S(\m/\mathfrak{n})$ is equal to $1$.
It is easy to check that $\m/\mathfrak{n}$ is an $R$-module.
And we have a calculation $\ell_R(\m/\mathfrak{n})\times \ell_S(R/\m)=\ell_S(\m/\mathfrak{n})=1$.
Thus it follows that $\m/\mathfrak{n}$ has dimension one as a vector space over $R/\m$.
Fix a preimage $t\in R$ of a basis $\overline{t}$ of $\m/\mathfrak{n}$.
Then $\m=\mathfrak{n}+Rt$ and $\m^2=\mathfrak{n}^2+\m t\subseteq \mathfrak{n}$.
This means $\m\subseteq S:\m$.
We have another inclusion $S:\m\subseteq S:\mathfrak{n}$.
Using Claim 2 in the proof of Lemma \ref{lem22}, we see that $R=\n:\n=S:\n$.
It also holds that $Rt\not\subset S$ (otherwise $\m=\n+Rt\subseteq S$).
These yield that $S:\m=\m$.
The fractional ideal $S:\m$ is isomorphic to $\Hom_S(\m,S)\cong\Hom_R(\m,\Hom_S(R,S))$.
Now as $S$ is Gorenstein and $S\subseteq R$ is a local homomorphism which makes $R$ a finite $S$-module, $\omega$ is isomorphic to $\Hom_R(R,S)$ \cite[Theorem 3.3.7 (b)]{BH}.
Thus $S:\m$ is isomorphic to $\Hom_R(\m,\omega)=\m^\dag$.
We conclude that $\m\cong\m^\dag$.

(3) $\implies$ (4): Applying the functor $(-)^\dag$ to the short exact sequence $0\to \m \to R \to k \to 0$, we see that the resulting exact sequence is $0 \to \omega \to \m^\dag \to \Ext^1_R(k,\omega)\cong k \to 0$.
Replacing $\m^\dag$ by $\m$, using the assumption $\m\cong \m^\dag$, we get the desired exact sequence.

(4) $\implies$ (3): Applying the functor $(-)^\dag$ to the short exact sequence $0\to \omega \to \m \to k \to 0$, we get an exact sequence $0 \to \m^\dag \to R \to \Ext^1_R(k,\omega)\cong k \to 0$.
Then, the image of $\m^\dag$ in $R$ must be equal to $\m$ and hence one has an isomorphism $\m^\dag\cong \m$.

(4) $\implies$ (5): The exact sequence $0 \to \omega \to \m \to k \to 0$ yields that there is an ideal $I\cong \omega$ such that the colength $\ell_R(\m/I)$ is one.
The equality $\ell_R(R/I)=2$ immediately follows from the above.

(5) $\implies$ (4): Take an ideal $I\cong \omega$ such that $l(R/I)\le 2$.
If $I=R$, then $R$ is Gorenstein and there is nothing to prove.
So we may suppose that $I\subseteq \m$.
If $I=\m$, then $\m\cong \omega$.
Take a regular element $x\in \m\setminus \m^2$ of $R$.
Then $\omega/x\omega \cong \m/x\m$, and $\m/x \m$ is isomorphic to $k\oplus \m/(x)$ (see \cite[Lemma 2]{Vas} for instance).
On the other hand, $\omega/x \omega$ is a canonical module of $R/(x)$.
Thus $\omega/x\omega$ is indecomposable.
This implies that $\m/(x)=0$, i.e. $\m=(x)$.
In particular, $R$ is a discrete valuation ring.

Now we deal with an assumption that $I\subsetneq \m$.
The inequality $l(R/I)\le 2$ implies that the equality $l(\m/I)=1$.
Thus the exact sequence $0\to I \to \m \to k \to 0$ is induced.
\end{proof}

All that remains is to show the direction (5) $\implies$ (1).
Let $(R,\m)$ be a Noetherian local ring containing a coefficient field $k\cong R/\m$.
Let $I\subset R$ be a fractional ideal such that $\ell_R(R/I)<\infty$. 
Put $k+I\coloneqq\{a+b\mid a\in k, b\in I\}\subseteq R$, which is a $k$-subalgebra of $R$.
Then, since $\dim_k R/(k+I)\leq \ell_R(R/I)<\infty$, $R$ is finitely generated as a $k+I$-module and hence $k+I$ is Noetherian by Hilbert basis theorem.
By the lying over property of $k+I\subseteq R$ (see \cite[Proposition 4.15]{E} for instance), any maximal ideal of $k+I$ is contained in $\m$.
Therefore $(k+I)\cap \m=I$ is the unique maximal ideal of $k+I$.
It also holds that $k+I$ contains a regular element in its maximal ideal $I$.
Since we have inequalities
\begin{align*}
\ell_{k+I}(R/(k+I))&=\ell_{k+I}(R/I)-\ell_{k+I}((k+I)/I)\\
&=\ell_{k+I}(R/I)-1\\
=\ell_R(R/I)\ell_{k+I}(R/\m)<\infty,
\end{align*}
$R/(k+I)$ is torsion $k+I$-module.
Thus $R/(k+I)\otimes_{k+I} Q(k+I)=0$.
This implies that $R\otimes_{k+I} Q(k+I)=Q(k+I)$, equivalently $Q(R)=Q(k+I)$. 
Consequently the ring extension $k+I\subseteq R$ is module-finite residually rational and birational. 

\begin{lem} \label{lmlmlm}
Let $(R,\m)$ be a one-dimensional Cohen-Macaulay local ring.
Assume $R$ has a canonical ideal $I\cong \omega$ such that $l(R/I)=2$.
Put $S=k+I$.
Then $S$ is Gorenstein, and the colength $\ell_S(R/S)$ is equal to $1$.
\end{lem}

\begin{proof}
	$S$ is local with a maximal ideal $\mathfrak{n}=I$.
	The extension $S\subseteq R$ is module-finite, residually rational and birational.
	Since $I$ is a canonical ideal, we have $I:I=R$.
	Equivalently, $\mathfrak{n}:\mathfrak{n}=R$.
	In particular, the colength $\ell_S(R/S)$ is equal to the Cohen-Macaulay type of $S$ (Lemma \ref{lem21}).
	Since $R$ and $S$ have same residue field $k$, we can see the equalities $\ell_S(R/S)=l_S(\m/\mathfrak{n})=\ell_R(\m/I)$.
	On the other hand, we have
	\[
	\ell_R(\m/I)=\ell_R(R/I)-\ell_R(R/\m)=2-1=1.
	\]
	It follows that $S$ has Cohen-Macaulay type $1$, that is, $S$ is Gorenstein.
	Moreover, the colength $\ell_S(R/S)$ is equal to $\ell_R(\m/I)=1$.
\end{proof}

\begin{proof}[Proof of Theorem \ref{thA} (5) $\implies$ (1)]
Assume there is a canonical ideal $I$ such that $\ell_R(R/I)\le 2$.
If $\ell_R(R/I)\le 1$, then $I=R$ or $\m$.
In both of these cases, $R$ should be Gorenstein (in the case of $I=\m$, see the proof of Theorem 1.4 (5)$\implies$ (4)).

Thus we only need to consider the case $\ell_R(R/I)=2$.
By previous lemma, the ring $S\coloneqq k+I$ is Gorenstein and the colength $\ell_S(R/S)$ is $1$.
This shows $bg(R)\le 1$.
\end{proof}

We put the following lemma here, which will be used in the proof of Corollary \ref{co1}.

\begin{lem} \label{lem16}
Let $(R,\m)$ be a Cohen-Macaulay generically Gorenstien local ring of dimension one having a canonical module.
Assume $R$ is not a discrete valuation ring.
Then
\begin{itemize}
\item[(1)] $R$ has minimal multiplicity if and only if $\m\cong \m:\m$.
\item[(2)] $R$ is almost Gorenstein in the sense of \cite{GMP} if and only if $\m^\dag\cong\m:\m$.
\end{itemize}
\end{lem}

\begin{proof}
	See \cite[Proposition 2.5]{Ooishi} and \cite[Theorem 2.14]{Ko} respectively.
\end{proof}

We give a proof of Corollary \ref{co1} as follows.

\begin{proof}[Proof of Corollary \ref{co1}]
The implications (3) $\iff$ (4) $\iff$ (5) follow immediately from Lemma \ref{lem16}.

(1) $\iff$ (2): In the case $bg(R)=0$, $R$ is Gorenstein and has minimal multiplicity, and thus $e(R)\leq 2$ (see \cite[3.2. Corollary]{Sally2}).
The converse also holds.
Now suppose $bg(R)=1$.
Then by Theorem \ref{thA}, $\m$ is isomorphic to $\m^\dag$.
Therefore, $R$ has minimal multiplicity if and only if $R$ is almost Gorenstein.

(1) $\implies$ (3): Clear.

Now assume the residue field $R/\m$ is infinite.

(6) $\implies$ (4): Obviously, $S/\mathfrak{n}$ is also infinite.
If $e(S)\leq \edim S$, then $e(S)\leq 2$.
Using \cite[Theorem A.29 (iii)]{LW}, we see that 
\begin{align*}
e(R) &=\max\{\ell_R(X/\m)\mid R\subseteq \text{is a finite birational extension}\}\\
&\le \max\{\ell_S(X/\n)\mid S\subseteq \text{is a finite birational extension}\}=e(S),
\end{align*}
and so we have an inequality $e(R)\leq 2$.
This says that $R$ is Gorenstein and has minimal multiplicity.
So we may assume $e(S)=\edim S+1$.

Take a minimal reduction $(t)$ of $\mathfrak{n}$ and a preimage $\delta\in \mathfrak{n}^2$ of a generator of the socle of $S/(t)$.
Then $\mathfrak{n}^3=t\mathfrak{n}^2$, $\ell_S(\n^2/t\n)=1$ (see \cite[Proof of (3.4)]{Sally2}) and $(t):_S \mathfrak{n}=(t)+S\delta$.
The equality $\ell_S(\n^2/t\n)=1$ implies $\mathfrak{n}^2=t\mathfrak{n}+S\delta$.

Now we claim the following.

\begin{claim} \label{cla3}
$\End_S(\mathfrak{n})\cong (t):_S \mathfrak{n}/t$.
\end{claim}

\begin{proof}[Proof of the claim]
Recall that $\End_S(\mathfrak{n})$ is isomorphic to $\n:\n$.
We want to show the equality $\n:\n=(t):_S \mathfrak{n}/t$ as subsets of $Q(S)$.
The containment $((t):_S \n)\n\subseteq (t)$ shows that $((t):_S \n)/t$ is contained in $S:\n$.
By Lemma \ref{lem21}, $S:\n$ coincides with $\n:\n$.
In particular, the inclusion "$\supseteq$" holds.

Since $t$ is in $\n$, we have $t(\n:\n)\subseteq \n \subset S$.
Thus the inequality $t(\n:\n)\n\subseteq t\n\subseteq (t)$ shows that the inclusion "$\subseteq$".
\end{proof}

Therefore 
\[R\cong \End_S(\mathfrak{n})\cong (t):_S \mathfrak{n}/t=S+S(\delta/t).
\]
Identify $R$ with $S+S(\delta/t)$.
Since $R$ is local and $\delta^2\in\mathfrak{n}^4=t^2\mathfrak{n}^2$, $(\delta/t)$ cannot be a unit of $R$.
This shows $\m=\mathfrak{n}+S(\delta/t)$.
By this equality, we also have an isomorphism $R/\m\cong S/\mathfrak{n}$ induced by $S\subseteq R$. Observe the following equalities
\[
t\m=t\mathfrak{n}+S\delta=\mathfrak{n}^2
\]
and
\[
\m^2=(\mathfrak{n}+S(\delta/t))^2=\mathfrak{n}^2+\mathfrak{n}(\delta/t)+S(\delta/t)^2.
\]
Then $\delta^2\in\mathfrak{n}^4=t^2\mathfrak{n}^2$ implies $S(\delta/t)^2\subseteq \mathfrak{n}^2$, and $\mathfrak{n}\delta\subseteq \mathfrak{n}^3=t\mathfrak{n}^2$ implies $\mathfrak{n}(\delta/t)\subseteq \mathfrak{n}^2$. So $\m^2=\mathfrak{n}^2=t\m$. This means that $R$ has minimal multiplicity. 

It remains to show that $\m\cong \m^\dag$.
By Theorem \ref{thA}, it holds that either $R$ is Gorenstein or $\m\cong \m^\dag$.
In the case that $R$ is Gorenstein, it holds that $e(R)\le 2$ and so $\m$ is self-dual by \cite[Theorem 2.6]{Ooishi}.


(5) $\implies$ (7): Assume $R$ is almost Gorenstein and has minimal multiplicity.
Then we already saw that $\m$ is self-dual (Lemma \ref{lem16}).
It follows from \cite[Theorem 3.16]{GMP} that $\rho(R)\le 2$.

(7) $\implies$ (5): 
Recall that $\rho(R)$ is the reduction number of a canonical ideal of $R$ (\cite[Definition 4.2]{GGHV}).
So if $\rho(R)\le 1$, then $R$ is Gorenstein (\cite[Theorem 3.7]{GMP}).
It means that $R\cong \omega$.
We may assume that $R$ is not a discrete valuation ring.
Therefore $\m:\m=R:\m$ by Lemma \ref{lem21}, and so we have $\m:\m\cong \Hom_R(\m,R)\cong \Hom_R(\m,\omega)=\m^\dag$.
Since $\m$ is self-dual, this yields that $\m:\m\cong \m$.
Using Lemma \ref{lem16}, we deduce that $R$ has minimal multiplicity.

Assume $\rho(R)=2$.
Combining \cite[Theorem 3.5 (b), Proposition 3.8]{DR} and Theorem \ref{thA}, we obtain that $R$ is almost Gorenstein and has minimal multiplicity.

Finally, we deal with the assumption that $R$ contains a infinite field $k$ isomorphic to $R/\m$ via $R\to R/\m$.

(4) $\implies$ (1): This follows directly from Theorem \ref{thA}.

(1) $\implies$ (6): First we consider the case that $R$ is Gorenstein (i.e.\ $bg(R)=0$).
In this case, $e(R)\le 2$ and $\edim R\le 2$ by the assumption.
Take a minimal reduction $Rt$ of $\m$.
Put $I=Rt$ and $S=k+I$.
Then $\m^2=t\m$.
In particular, $\ell_R(\m/I)=\ell_R(\m/I+\m^2)\le 1$.
Then the ring-extension $S\subseteq R$ is module-finite, residually rational and birational.
Since $I:I=R$ and $\ell_R(\m/I)\le 1$, we can see that $S$ is Gorenstein and $\End_S(I)\cong R$ by the similar argument in the proof of \ref{thA} (3) $\implies$ (1).
Furthermore, one has an equality $tI=I^2$, which particularly show that $S$ has minimal multiplicity, that is, $e(S)=\edim S$.

Now consider the case that $bg(R)=1$.
Repeating the proof of Theorem \ref{thA} (3) $\implies$ (1), there is a canonical ideal $I$ such that if we let $S=k+I$, then $S$ is Gorenstein local and $R\cong \End_S(\mathfrak{n})$, where $\mathfrak{n}$ is the maximal ideal of $S$.
Since $R$ is Almost Gorenstein, it was shown in \cite[Theorem 3.16]{GMP} that there is a minimal reduction $Q=(t)\subseteq I$ of $I$ in $R$ such that $\ell_R(I^2/QI)\le 1$ and $QI^2=I^3$.
Then it follows that $\ell_S(I^2/QI)\le 1$.
Using \cite[Proposition 3.3]{Sally2}, the equality $e(S)=\edim S+1$ holds.
\end{proof}

We give here an example of a ring $R$ with $bg(R)=1$.

\begin{ex}
	Let $R=k[[t^3,t^4,t^5]]$ and $S=k[[t^3,t^4]]$ be numerical semigroup rings, where $k$ is a field.
	Then the natural inclusion $S\subseteq R$ is a module-finite birational extension of local rings with the same coefficient field.
	The colength $\ell_S(R/S)$ is equal to $1$.
	Since $R$ is non-Gorenstein and $S$ is Gorenstein, we have $bg(R)=1$.
\end{ex}

We now turn to estimate the invariant $bg(R)$ in general.
Suppose there exists a self-dual fractional ideal of $R$.
Then we have an upper bound of $bg(R)$ as follows.

\begin{lem} \label{lem23}
Let $(R,\m)$ be a complete one-dimensional Cohen-Macaulay local ring.
Assume $R$ contains an infinite coefficient field $k\cong R/\m$.
Let $I\subseteq R$ be a fractional ideal of $R$. If $I$ is self-dual, then we have $bg(R)\leq l(R/I)$.
In other words, the following inequality holds
\[
bg(R)\leq \inf\{\ell_R(R/I)\mid I\cong I^\dag \}.
\]
\end{lem}

\begin{proof}
In the case $I=R$, the self-duality of $I$ implies $R$ is Gorenstein.
So we may assume $I\subseteq \m$.
Take a non-zerodivisor $t\in I$, and Put $B=k+I$.
Then $B\subseteq R$ is a module-finite extension and $I$ is the maximal ideal of local ring $B$.
Remark that $B$ is also complete, and so a canonical module $\omega_B$ of $B$ exists.
Since $B\subseteq R$ is birational, the $R$-isomorphism $I\to I^\dag$ is also a $B$-isomorphism.
We also have an isomorphism $\Hom_B(R,\omega_B)\cong \omega_R$, which yields isomorphisms
\[
\Hom_B(I,\omega_B)\cong\Hom_B(I\otimes_R R,\omega_B)\cong \Hom_R(I,\Hom_B(R,\omega_B))\cong \Hom_R(I,\omega_R)\cong I^\dag.
\]

This says that $I^\dag$ is isomorphic to the canonical dual of $I$ over $B$.
By Theorem \ref{thA}, $bg(B)\leq 1$, that is, there is a Gorenstein ring $S$ and module-finite birational extension $S\subseteq B$.
Then $S\subseteq R$ is also a module-finite birational extension.
The calculation
\[
\ell_S(R/S)=\ell_S(R/B)+\ell_S(B/S)=\ell_R(\m/I)+1=\ell_R(R/I)
\]
shows that $bg(R)\leq \ell_R(R/I)$.
\end{proof}

As a corollary of this, we can see the finiteness of $bg(R)$ in the analytically unramified case.

\begin{cor}
Let $(R,\m)$ be a complete one-dimensional local ring.
Assume $R$ contains an infinite coefficient field.
If there exists a module-finite birational extension $R\subseteq T$ with a Gorenstein ring $T$, Then $bg(R)\leq l(R/aT)$ for any non-zerodivisor $a\in T:R$ of $T$.
Moreover, if $R$ is analytically unramified, then $bg(R)\leq l(R/R:\widetilde R)<\infty$, where $\widetilde R$ is the integral closure of $R$ in $Q(R)$.
\end{cor}

\begin{proof}
Since $T$ is Gorenstein, the $R$-module $aT\cong T$ is self-dual.
So we can apply Lemma \ref{lem23} for $I=aT$.
If $R$ is analytically unramified, then $\widetilde R$ is Gorenstein, and $R\subseteq \widetilde R$ is finite birational.
The conductor $R:\widetilde R$ is a nonzero and satisfies $(R:\widetilde R)\otimes_R Q(R)=Q(R)$.
Thus $R:\widetilde{R}$ has constant rank one and contains a non-zerodivisor of $R$.
In particular, $R:\widetilde{R}$ is torsion-free over $\widetilde{R}$.
As $\widetilde{R}$ is reduced and integrally closed in its total ring of quotients, its localization at any maximal ideal $\p$ is a discrete valuation ring.
Therefore $(R:\widetilde{R})_\p$ is a free module of rank one for any $\p$.
Since $\widetilde{R}$ is semilocal, it follows that $R:\widetilde{R}$ is free of rank one over $\widetilde{R}$.
This means that $R:\widetilde{R}\cong \widetilde{R}$.
Applying Lemma \ref{lem23} for $I=R:\widetilde{R}$, we have an inequality $bg(R)\leq l(R/R:\widetilde R)<\infty$.
\end{proof}

\begin{rem}
Ananthnarayan \cite{Ananthnarayan} shows the following inequalities hold for an artinian local ring $R$.

\begin{equation}
\ell_R(R/\omega^*(\omega))\leq \min\{\ell_R(R/I)\mid I\cong I^\dag\} \leq g(R).
\end{equation}
Here $\omega^*(\omega)$ is the trace ideal of $\omega$; see Definition \ref{deftr}.
\end{rem}

As analogies of these inequalities, the followings are natural questions.

\begin{ques} \label{Q1}
Let $(R,\m)$ be a one-dimensional Cohen-Macaulay local ring. 
Does an inequality 
\[
bg(R)\geq\inf\{\ell_R(R/I)\mid I\cong I^\dag\}
\] hold true?
\end{ques}

\begin{ques} \label{Q2}
Let $(R,\m)$ be a one-dimensional generically Gorenstein local ring.
Does an inequality $\ell_R(R/\omega^*(\omega))\leq bg(R)$ hold true?
\end{ques}

By our main theorems \ref{thA} and \ref{thB}, Question \ref{Q1} is affirmative for $R$ with $bg(R)\leq 2$.
Question \ref{Q2} has positive answer given in Proposition \ref{p1} if $bg(R)\leq 1$.

We now return to prove the Theorem \ref{thB}.

\begin{proof}[Proof of Theorem \ref{thB}]
(2) $\implies$ (1): This is a consequence of Lemma \ref{lem23} by letting $I=\mathfrak{a}$.

(1) $\implies$ (2): In the case $bg(R)\leq 1$, assertion follows by Theorem \ref{thA}.
So we may assume $bg(R)=2$.
Take a Gorenstein local ring $(S,\mathfrak{n})$ and module-finite residually rational birational extension $S\subset R$ satisfying $\ell_S(R/S)=2$. 
Note that, since $S$ is Gorenstein and $R$ is a maximal Cohen--Macaulay $S$-module, $R$ is reflexive.

Let $B$ be the ring $\mathfrak{n}:\mathfrak{n}$.
By Lemma \ref{lem22} and Lemma \ref{lem21}, we have $\ell_S(B/S)=1$ and $S\subsetneq B\subseteq R$.
Therefore $\ell_B(R/B)=1$.
As in the paragraph before Lemma \ref{lmlmlm}, the lying over property of the extension $B\subseteq R$ shows that $B$ is local.
Let $\m_B$ be the maximal ideal of $B$ and fix a preimage $t\in R$ of a basis $\overline{t}$ of the one-dimensional vector space $R/B$ over $B/\m_B$. By the relation $\m_Bt\subseteq B$ yields that $t\in B:\m_B=\m_B:\m_B$. Therefore $R=B+Bt\subseteq \m_B:\m_B$.
In particular, $R\m_B\subseteq \m_B$.
This says that $\m_B$ is an ideal of $R$.
Since $bg(B)=1$, $\m_B$ is a self-dual ideal of $B$ by Theorem \ref{thA}.
Fix a canonical module $\omega_B$ of $B$.
Then $\Hom_B(B,\omega_B)$ is a canonical module of $R$.
\[
\Hom_B(\m_B,\omega_B)\cong\Hom_B(\m_B\otimes_R R,\omega_B)\cong \Hom_R(\m_B,\Hom_B(R,\omega_B))\cong \Hom_R(\m_B,\omega_R).
\]

Thus, it is also self-dual as an $R$-module.
One can also have equalities
\[
\ell_R(R/\m_B)=\ell_B(R/B)+\ell_B(B/\m_B)=2.
\]
\end{proof}


\begin{rem}
Let $(R,\m)$ be a one-dimensional local ring. Assume $R$ is complete, equicharacteristic and $bg(R)=n<\infty$. If there exists a Cohen-Macaulay local ring $(B,\m_B)$ with $bg(B)=1$ and module-finite residually rational birational extensions $B\subseteq R\subseteq \m_B:\m_B$ such that $\ell_B(R/B)+1\leq n$. Then, by the same argument of proof of Theorem \ref{thB}, it follows that $\m_B$ is a self-dual ideal of $R$ satisfying $\ell_R(R/\m_B)\leq n$. In this case, Question \ref{Q1} is affirmative for $R$.
\end{rem}

\section{The self-duality of the maximal ideal}

In this section, we collect some properties of local rings $(R,\m)$ with $\m\cong \m^\dag$.

\begin{lem} 
Let $(R,\m)$ be a Cohen-Macaulay local ring with a canonical module. Assume $\m\cong \m^\dag$. Then 
\begin{itemize}
\item[(1)] $\dim R\leq 1$.
\item[(2)] Let $x\in\m\setminus \m^2$ be a non-zerodivisor of $R$. Then $R/(x)$ also has self-dual maximal ideal.
\item[(3)] $\edim(R)=r(R)+\dim R$.
\end{itemize}
\end{lem}

\begin{proof}
(1) Suppose $\dim R\geq 2$ and $\omega$ is a canonical module of $R$.
Applying $(-)^\dag$ to the exact sequence $0 \to \m \to R \to k \to 0$, we get an exact sequence
\[
0 \to \Hom_R(k,\omega) \to \omega \to \m^\dag \to \Ext_R^1(k,\omega).
\]
By the assumption $\dim R\geq 2$ yields that $\Hom_R(k,\omega)=0=\Ext_R^1(k,\omega)$ and hence $\m^\dag \cong \omega$.

From the isomorphism $\m\cong \m^\dag$, it follows that $\m \cong R$, i. e. $\m$ is a principal ideal.
This shows that $\dim R\leq 1$, which is a contradiction.
Thus, it must be $\dim R\leq 1$.

(2) Applying the functor $\Hom_R(-,\omega)$ to the exact sequence $0\to \m \xrightarrow[]{x} \m \to \m/x\m \to 0$, we get an exact sequence 
\[
0\to \Hom_R(\m/x\m,\omega) \to \m^\dag \xrightarrow[]{x} \m^\dag \to \Ext^1_R(\m/x\m,\omega) \to \Ext^1_R(\m,\omega).
\]
Since $\dim R$ is less than or equal to $1$ by (1), and $x\in \m$ is a non-zerodivisor, it follows that $\dim R=1$.
Thus $\m$ is a maximal Cohen--Macaulay $R$-module, which yields the equality $\Ext^1_R(\m,\omega)=0$.
It also follows that $\Hom_R(\m/x\m,\omega)=0$ and $\Ext^1_R(\m/x\m,\omega)\cong \Hom_{R/(x)}(\m/x\m,\omega/x\omega)$ (see \cite[Lemma 1.2.4]{BH}).
Thus we get an isomorphism $\m^\dag/x\m^\dag\cong \Hom_{R/(x)}(\m/x\m,\omega/x\omega)$.
From this isomorphism and $\m\cong \m^\dag$, isomorphisms $\m/x\m\cong \m^\dag/x\m^\dag\cong \Hom_R/(x)(\m/x\m,\omega/x\omega)$ are induced.
By \cite[Lemma 2]{Vas}, we have an isomorphism $\m/x\m\cong R/\m\oplus \m/(x)$.
Therefore we obtain isomorphisms
\[
R/\m\oplus \m/(x)\cong \Hom_{R/(x)}(R/\m\oplus \m/(x),\omega/x\omega)\cong \Hom_{R/(x)}(R/\m,\omega/x\omega)\oplus \Hom_{R/(x)}(\m/(x),\omega/x\omega).
\]
Remark that $\omega/x\omega$ is a canonical module of $R/(x)$ (\cite[Theorem 3.3.5]{BH}).
Thus $\Hom_R/(x)(R/\m,\omega/x\omega)\cong R/\m$.
Then it follows that 
$$R/\m\oplus \m/(x) \cong R/\m \oplus \Hom_R/(x)(\m/(x),\omega/x\omega).$$
By the Krull--Schmidt theorem for $R/(x)$, this yields that $\m/(x)\cong \Hom_R/(x)(\m/(x),\omega/x\omega)$, which means the self-duality of the maximal ideal $\m/(x)$ of $R/(x)$.

(3) Suppose that $\dim R=0$.
Then by $\m\cong \m^\dag$ and \cite[Proposition 3.3.11]{BH}, $\mu(\m)=\mu(\m^\dag)=r(\m)$.
Here for an $R$-module $X$, $\mu_R(X)$ denotes the minimal number of generators of $X$.
Since $\dim \m=\dim R=0$, $r(\m)$ (resp. $r(R)$) is equal to $\dim_{R/\m}(\mathrm{soc} \m)$ (resp. $\dim_{R/\m}(\mathrm{soc} R)$), and so the equality $\mathrm{soc} \m=\mathrm{soc} R$ implies that $r(\m)=r(R)$.
Thus we have $\edim R=\mu(\m)=r(\m)=r(R)+\dim R$.

Now suppose that $\dim R>0$.
Then (1) shows that $\dim R=1$.
Since $R$ is Cohen--Macaulay, we can take a non-zerodivisor $x\in \m\setminus \m^2$ of $R$.
Then, by (2), $R/(x)$ has self-dual maximal ideal.
Since $\dim R/(x)=0$, it follows that $\edim R/(x)=r(R/(x))$.
Note that $r(R/(x))=r(R)$ and $\edim R/(x)=\mu_{R/(x)} \m/(x)=\mu_R(\m)-1$.
Thus we have equalities
$$\edim R=\mu(\m)=\edim R/(x)+1=r(R)+1$$.

\end{proof}

When $\dim R\ge 2$, the maximal ideal $\m$ cannot be self-dual.
However, we suggest the following generalization of the self-duality of the maximal ideal in higher dimensional case.

\begin{prop}
	$(R,\m)$ be a non-Gorenstein Cohen-Macaulay local ring of dimension $d>0$ having an infinite residue field.
	Assume $R$ has a canonical ideal $I$ satisfying $e(R/I)=2$.
	Then there is a regular sequence $\underline{x}=x_1,\dots,x_{d-1}$ such that $R/(\underline{x})$ has self-dual maximal ideal.
\end{prop}

\begin{proof}
	Since $R/I$ is Cohen-Macaulay of dimension $d-1$ (\cite[Proposition 11.6]{LW}), we can take a minimal reduction $\underline{y}=y_1,\dots,y_{d-1}$ of the maximal ideal $\m/I$ in $R/I$.
	Then the length $\ell((R/I)/(\underline{y}))$ is equal to $e(R/I)(\le2)$.
	Let $\underline{x}=x_1,\dots,x_{d-1}$ be a preimage of $\underline{y}$ in $R$.
	As $I$ is unmixed, we can take $\underline{x}$ as a regular sequence in $R$.
	The tensor product $I'=I\otimes R/(\underline{x})$ is naturally isomorphic a canonical ideal of $R'=R/(\underline{x})$.
	The quotient $R'/I'$ has length $\ell(R/(I+\underline{x}))=\ell((R/I)/(\underline{y}))\le 2$.
	Therefore $R'$ has self-dual maximal ideal by Theorem \ref{thA}.
\end{proof}

\begin{ex}
	Let $R=k[[x^3,x^2y,xy^2,y^3]]$ be the third Veronese subring of $k[[x,y]]$.
	Then $I=(x^3,x^2y)R$ is a canonical ideal of $R$.
	The quotient ring $R/I$ is isomorphic to $k[[s,t]]/(s^2)$, and hence $e(R/I)=2$.
\end{ex}

Go back to the subject on self-duality of the maximal ideal.
Recall the notion of trace ideal of an $R$-module and nearly Gorensteiness of local rings (see \cite{HHS}).

\begin{dfn} \label{deftr}
Let $R$ be a commutative ring. For an $R$-module $M$, the {\it trace ideal} $M^*(M)$ of $M$ in $R$ is defined to be the ideal $\sum_{f\in \Hom_R(M,R)} \image f\subseteq R$.

A Cohen-Macaulay local ring $(R,\m)$ with a canonical module $\omega$ is called {\it nearly Gorenstein} if $\omega^*(\omega)\supseteq \m$.
\end{dfn}

\begin{lem} \label{lemng}
Let $(R,\m)$ be a Cohen-Macaulay local ring with a canonical module.
The following are equivalent.
\begin{itemize}
\item[(1)] $R$ is nearly Gorenstein.
\item[(2)] there is a surjective homomorphism $\omega^{\oplus n} \to \m$ for some $n$.
\end{itemize}
Moreover, if $\dim R\leq 1$, then we can add the following conditions.
\begin{itemize}
\item[(3)] there is a short exact sequence $0 \to \m^\dag \to R^{\oplus n} \to M \to 0$ for some $n$ and maximal Cohen-Macaulay module $M$.
\item[(4)] there is a short exact sequence $0 \to \m^\dag \to \m^{\oplus n} \to M \to 0$ for some $n$ and maximal Cohen-Macaulay module $M$.
\end{itemize}
\end{lem}

\begin{proof}
(1) $\iff$ (2): By the definition of trace ideals, there is a surjection $\omega^{\oplus n}\to \omega^*(\omega)$ for some $n$.
So the equivalence immediately follows.
	
Now assume $\dim R\le 1$.
Then the maximal ideal $\m$ is maximal Cohen-Macaulay as an $R$-module.
So the condition (2) is equivalent to that there is a short exact sequence $0\to M\to \omega^{\oplus n}\to \m \to 0$ for some $n$ and maximal Cohen-Macaulay module $M$.
Taking the canonical duals, the equivalence of (2) and (3) follows.
	
We turn the equivalence of (3) and (4).
We may assume $R$ is not a discrete valuation ring, and hence both $\m$ and $\m^\dag$ are not free $R$-modules.
Assume that (3) holds.
The condition (3) means that $\m^\dag$ is a syzygy of a maximal Cohen--Macaulay module.
Thus by \cite[Lemma 2.8]{Ko}, there is a short exact sequence $0\to \m^\dag \to \m^{\oplus n} \to M' \to 0$ with some maximal Cohen--Macaulay $R$-module $M'$.
This shows the implication (3) $\implies$ (4).

Conversely, suppose that (4) holds.
Then we may use \cite[Lemma 2.6]{Ko} to show that (3) holds, since $\m$ is a syzygy of a maximal Cohen--Macaulay module by \cite[Lemma 2.7]{Ko}.
\end{proof}

\begin{prop} \label{p1}
Let $(R,\m)$ be a Cohen-Macaulay local ring with a canonical module. Assume $\m\cong \m^\dag$. Then
\begin{itemize}
\item[(1)] $R$ is nearly Gorenstein.
\item[(2)] If $R$ is non-Gorenstein and $2$ is invertible in $R$, then $R$ is G-regular.
\end{itemize}
\end{prop}

\begin{proof} We already saw that $\dim R\leq 1$ from Lemma \ref{lemng}.

(1) In the case of $\dim R=0$, we have a short exact sequence $0\to \m \to R \to k \to 0$ and hence we can apply Lemma \ref{lemng} (3) $\implies$ (1).

In the case of $\dim R=1$, we may assume $R$ is not a discrete valuation ring.
Applying Lemma \ref{lemng} to the short exact sequence in Theorem \ref{thA} (4), it follows that $R$ is nearly Gorenstein.

(2)  In the case that $\dim R=0$, the statement is proved in \cite[Corollary 3.4]{Striuli-Vraciu}.
So we may assume $\dim R=1$.
Take $x\in \m\setminus \m^2$ a non-zerodivisor.
Thanks to Lemma 3.1, the maximal ideal of $R/(x)$ is self-dual.
Then $R/(x)$ is G-regular by \cite[Corollary 3.4]{Striuli-Vraciu}.
It follows from \cite[Proposition 4.2]{Tak} that $R$ is also G-regular.
\end{proof}

\begin{ex}
Let $R=k[[t^4,t^5,t^7]]$. Then $R$ is almost Gorenstein local ring of dimension one. Therefore, $R$ is G-regular and nearly Gorenstein. On the other hand, $R$ does not have minimal multiplicity, and hence $\m$ is not self-dual. This shows that the converse of Proposition \ref{p1} doesn't hold true in general. 
\end{ex}

\begin{ex}
	The associated graded ring $\mathsf{gr}_\m(R)$ of a local ring $(R,\m)$ with self-dual maximal ideal need not be Cohen-Macaulay, for example, $R=k[[t^4,t^5,t^{11}]]$.
\end{ex}

We use the notion of minimal faithful modules.
The definition of them is given in below.

\begin{dfn}
Let $R$ be a commutative ring. An $R$-module $M$ is called {\it minimal faithful} if it is faithful and no proper submodule or quotient module is faithful.
\end{dfn}

\begin{ex}
For an artinian local ring $R$, the $R$-module $R$ and a canonical module $\omega$ of $R$ (i.e. injective hull of the residue field) are minimal faithful.
\end{ex}

The following is proved by Bergman \cite[Corollary 2]{Bergman}.

\begin{lem}[Bergman] \label{bergm}
Let $A,B$ and $C$ be finite-dimensional vector spaces over a field $k$. and $f\colon A\times B\to C$ be a bilinear map.
Assume the following conditions.
\begin{enumerate}[\rm(1)]
\item
any nonzero element $a$ of $A$ induces a nonzero map $f(a,-)\colon B\to C$
\item 
For any proper submodule $i\colon B'\to B$, there is a nonzero element $a\in A$ such that $f(a,i(-))\colon B'\to C$ is a zero map.
\item
For any proper quotient module $p\colon C\to C'$ there is a nonzero element $a\in A$ such that the map $p\circ f(a,-)\colon B\to C'$ is a zero map.
\end{enumerate}
Then $\dim_k A\ge \dim_k B+\dim_k C -1$.
\end{lem}

To have an application of Lemma \label{bergm} for modules, we need the following lemmas.

\begin{lem}\label{injhom}
Let $R$ be a commutative ring, $n$ be a positive integer, $M,N$ be $R$-modules and $f=[f_1,\dots,f_n]^t\colon N\to M^{\oplus n}$ be an $R$-homomorphism.
Assume that $N$ is Artinian.
Then $f$ is injective if and only if for any nonzero element $a\in \mathrm{soc}(N)$, there exists $i$ such that $f_i(a)\not=0$.
\end{lem}

\begin{proof}
We can see that $\ker f=\ker f_1\cap\dots\cap \ker f_n$.
Thus $f$ is injective if and only if $\ker f_1\cap\dots\cap \ker f_n=0$.
Since $N$ is Artinian, the latter condition is equivalent to that $\mathrm{soc}(N)\cap\ker f_1\cap\dots\cap \ker f_n=0$.
\end{proof}

\begin{lem}\label{injdim}
Let $(R,\m,k)$ be an artinian local ring and $M, N$ be finitely generated $R$-modules.
Assume $M$ is minimal faithful and $N$ is Artinian.
If $n$ is the smallest positive integer such that exists an injective homomorphism $f=[f_1,\dots,f_n]\colon N \to M^{\oplus n}$, then the $k$-subspace $B$ of $\Hom_R(N,M)\otimes_R k$ generated by the image of $f_1,\dots,f_n$ has a dimension exactly equal to $n$ over $k$.
\end{lem}

\begin{proof}
We only need to show $\dim_k B\ge n$.
Assume there is a equation $f_1=a_1f_2+\cdots+a_nf_n+g$ for some $a_2,\dots,a_n\in R$ and $g\in \mathfrak{m}\Hom_R(N,M)$.
Then for any element $a\in \mathrm{soc}(N)$, $g(a)=0$.
So $n\ge 2$ and $f(a)\not=0$ implies there exists $i\ge 2$ such that $f_i(a)\not=0$.
This particular says that the homomorphism $[f_2,\dots,f_n]\colon N\to M^{\oplus n-1}$ also an injection by Lemma \ref{injhom}, which is a contradiction to our assumption on $n$.
\end{proof}

The following lemma is a generalization of the result of Gulliksen \cite[Lemma 2]{Gull}.

\begin{lem} \label{lemcmf}
Let $(R,\m,k)$ be an artinian local ring and $M, N$ be finitely generated faithful $R$-modules.
Assume $M$ is minimal faithful.
If there exists an injective homomorphism $N \to M^{\oplus n}$ for some $n$, then $\dim_k \mathrm{soc}(M)\leq\dim_k \mathrm{soc}(N)$ and equality holds if and only if $N\cong M$.
\end{lem}

\begin{proof}
Let $n$ be the minimal integer such that there is an injective map $N \to M^{\oplus n}$.
Take a injective map $N \xrightarrow[]{[f_1,\dots,f_n]^t} M^{\oplus n}$ and set $B$ the $k$-subspace of $\Hom_R(N,M)\otimes_R k$ generated by the image of $f_1,\dots,f_n$.
Then $\dim_k B=n$ by Lemma \ref{injdim}.
By letting $A=\mathrm{soc}(N)$ and $C=\mathrm{soc}(M)$, we have a bilinear map $A\times B \to C$ over $k$ satisfying the assumption (1) and (2) of Lemma \ref{bergm} in view of Lemma \ref{injhom} and \ref{injdim}.
We also verify the condition (3) of Lemma \label{bergm} as follows.
Assume (3) is not satisfied.
Then there is a subspace $C'$ of $C$ such that any nonzero element $a$ of $A$ induces a nonzero map $p\circ f(-,a)\colon B\to C/C'$, where $p\colon C\to C/C'$ is the natural surjection.
Since $C/C'\subseteq M/C'$ as an $R$-module, we obtain an injective map $g\colon N\xrightarrow[]{q\circ f_1,\dots,q\circ f_n}(M/C')^{\oplus n}$, where $q\colon M\to M/C'$ is also the natural surjection.
Since $N$ is faithful, there is an injective map $h$ from $R$ to some copies $N^{\oplus m}$ of $N$.
Taking a composition of $h$ and $g^{\oplus m}$, one has an injective map from $R$ to $(M/C')^{\oplus mn}$.
In particular, $M/C'$ is a faithful $R$-module, which contradicts the assumption that $M$ is minimal faithful.

Therefore, we can apply Lemma \ref{bergm} and get an equality $\dim A\geq \dim B+\dim C-1$.
It follows that $\dim \mathrm{soc}(N)-\dim \mathrm{soc}(M)\geq n-1\geq 0$.
If the equalities hold, then $n=1$ and $N$ is isomorphic to a submodule of $M$.
By the minimality of $M$, one has $N\cong M$.
\end{proof}


We also give some basic properties of minimal faithful modules.

\begin{lem} \label{lemmf}
Let $(R,\m,k)$ be an artinian local ring.
Then
\begin{enumerate}[\rm(1)]
\item Any minimal faithful $R$-module is indecomposable.
\item Assume $R$ has Cohen-Macaulay type at most three.
Then $\ell_R(R)\leq \ell_R(M)$ for all faithful $R$-module $M$.
In particular, a faithful $R$-module $M$ is minimal faithful if $\ell_R(M)=\ell_R(R)$.
\end{enumerate}
\end{lem}

\begin{proof}
(1): Let $M$ be a minimal faithful $R$-module, and assume that $M$ decomposes as direct sum $M=M_1\oplus M_2$ of $R$-modules.
the faithfulness of $M$ yields that $\ann(M_1)\cap\ann(M_2)=0$.
Take minimal generators $x_1,\dots,x_n$ of $M_1$ and $y_1,\dots,y_m$ of $M_2$.
Without loss of generality, we may assume $n\le m$.
Then the submodule $N$ of $M=M_1\oplus M_2$ generated by the elements $x_1+y_1,\dots,x_n+y_n,0+y_{n+1},\dots,0+y_m$ is proper and faithful.
This contradicts that $M$ is minimal faithful.
(2): This follows by \cite[Theorem 1]{Gull}.
\end{proof}

\begin{dfn}
Let $(R,\m,k)$ be a commutative ring. A fractional ideal $I$ of $R$ is called {\it closed} \cite{BV} if the natural homomorphism $R\to \Hom_R(I,I)$ is an isomorphism.
\end{dfn}

\begin{ex}
Let $(R,\m,k)$ be a one-dimensional Cohen-Macaulay local ring. Set $B=\m:\m$. Then $\m$ is closed as a fractional ideal of $B$.
\end{ex}

\begin{lem}
Let $(R,\m,k)$ be a one-dimensional Cohen-Macaulay local ring having a canonical module and $I$ be a fractional ideal of $R$. Then the following are equivalent.
\begin{itemize}
\item[(1)] $I$ is closed.
\item[(2)] $I^\dag$ is closed.
\item[(3)] There is a surjective homomorphism $I^{\oplus n}\to \omega$ for some $n$.
\item[(4)] There is a short exact sequence $0\to R \to I^{\oplus n} \to M \to 0$ for some $n$ and maximal Cohen-Macaulay $R$-module $M$.
\end{itemize}
\end{lem}

\begin{proof}
	See \cite[Proposition 2.1]{BV}.
	Note that (4) follows by the canonical dual of (3), since $I$ is maximal Cohen-Macaulay as an $R$-module.
\end{proof}


\begin{proof}[Proof of Theorem 1.7]
Take a minimal reduction $(t)$ of $\m_B$.
The assumption that $R$ is nearly Gorenstein implies that any localization of $\omega_R$ at a non-maximal prime ideal is free of rank one (\cite[Proposition 2.3]{HHS}).
This yields that $\omega_R\otimes_R Q(R)=Q(R)$.
On the other hand, $\m\otimes_B Q(B)=\m\otimes_R Q(R)=Q(R)=Q(B)$ since $\m$ contains a regular element.
Thus one has
\begin{align*}
\m^\dag\otimes_B Q(B)=\Hom_R(\m,\omega_R)\otimes_B Q(B)&=\Hom_R(\m,\omega_R)\otimes_R Q(R)\\
&=\Hom_{Q(R)}(\m\otimes_RQ(R),\omega_R\otimes_RQ(R))\\
&=\Hom_{Q(R)}(Q(R),Q(R))=Q(R)=Q(B).
\end{align*}

This shows that $\m$ and $\m^\dag$ has constant rank one as a $B$-module.
In particular, we have equalities $\ell_B(B/tB)=e(B)=\ell_B(\m/t\m)=\ell_B(\m^\dag/t\m^\dag)$.
Note that $\m/t\m$ and $\m^\dag/t\m^\dag$ are faithful over $B/tB$ (see \cite[Proposition 3.3]{BV}).
As $B/tB$ has Cohen-Macaulay type less than or equal to three in both case (1) and (2), Lemma \ref{lemmf} ensures that $\m/t\m$ and $\m^\dag/t\m^\dag$ are minimal faithful over $B/tB$.
Consider the exact sequence 
\[
0 \to \m^\dag \xrightarrow{\phi} \m^{\oplus n} \to M \to 0
\]
as in Lemma \ref{lemng}. Then $\phi\in \Hom_R(\m^\dag,\m^{\oplus n})=\Hom_B(\m^\dag,\m^{\oplus n})$. Therefore $M=\coker \phi$ is also a $B$-module and it is torsion-free over $B$ as well as over $R$. Moreover, the above sequence is an exact sequence of $B$-modules and $B$-homomorphisms.
Tensoring $B/tB$ to this sequence, we have a short exact sequence
\begin{equation} \label{seq}
0 \to \m^\dag/t\m^\dag \xrightarrow{\phi\otimes B/tB} (\m/t\m)^{\oplus n} \to M/tM \to 0
\end{equation}
of $B/tB$-modules. 

(1): Applying Lemma \ref{lemcmf} to the sequence (\ref{seq}) and using \cite[Lemma 2]{Gull}, we obtain the inequalities
\[
1\leq \dim_k \mathrm{soc}(\m/t\m) \leq \dim_k \mathrm{soc}(\m^\dag/t \m^\dag) \leq r_B(B)=2.
\]
So one has either $\dim_k \mathrm{soc}(\m/t\m)=1$ or $\dim_k \dim_k\mathrm{soc}(\m/t\m)=\dim_k\mathrm{soc}(\m^\dag/t \m^\dag)=2$.
In the former case, we have $1=r_B(\m)=\mu_B(\m^\dag)$, which shows that
$\m^\dag$ must be a cyclic $B$-module.
Thus $\m^\dag\cong B$.
This yields that $R$ is almost Gorenstein by Lemma 2.4.
Suppose that $\dim_k \mathrm{soc}(\m/t\m)=\dim_k\mathrm{soc}(\m^\dag/t \m^\dag)=2$.

Take a system of minimal generator $a_1,\dots,a_m$ of $\m/t \m$ over $B/tB$, and for  each $i=1,\dots m$ we set a homomorphism $f_i\colon B/tB \to \m/t \m$, which sends $1$ to $a_i$.
Since $\m/t\m$ is faithful over $B/tB$, The homomorphism $B/tB \xrightarrow[]{[f_1,\dots,f_m]^t}(\m/t\m)^{\oplus m}$ is injective.
Then by Lemma 3.14, the equality $\dim_k\mathrm{soc}(\m/t\m)=2=\dim_k\mathrm{soc}(B/tB)$ induces an isomorphism $\m/t\m\cong B/tB$.
This shows that $\m$ is a cyclic $B$-module.
Thus we have $\m\cong B$.
Similar argument shows that $\dim_k\mathrm{soc}(\m^\dag/t\m^\dag)=2=\dim_k\mathrm{soc}(B/tB)$ implies that $\m^\dag\cong B$.
Then $R$ is almost Gorenstein and has minimal multiplicity by Lemma 2.4, and so
$B$ has type one by Theorem 1.8, a contradiction.

(2): Applying Lemma \ref{lemcmf} to the sequence (\ref{seq}), we obtain the inequalities
\[
1\leq \dim_k \mathrm{soc}(\m/t\m) \leq \dim_k \mathrm{soc}(\m^\dag/t \m^\dag) \leq r_B(B)=3.
\]

In the case that $\dim_k \mathrm{soc}(\m/t\m)=1$ or $\dim_k \mathrm{soc}(\m^\dag/t \m^\dag)=3=\dim_k\mathrm{soc}(B/tB)$, it follows by same argument as in (1) that $\m^\dag\cong B$, which implies that $R$ is almost Gorenstein by Lemma 2.4.
So we only need to consider the case $\dim_k \mathrm{soc}(\m/t\m)=\dim_k \mathrm{soc}(\m^\dag/t \m^\dag)=2$.
In this case, $\m/t\m$ should be isomorphic to $\m^\dag/t\m^\dag$ by lemma \ref{lemcmf}.
Put $\phi=[\phi_1,\dots,\phi_n]^t\colon \m^\dag\to\m^{\oplus n}$ and so $\phi\otimes B/tB=[\phi_1\otimes B/tB,\dots,\phi_n\otimes B/tB]^t$.
Consider the canonical dual $(\phi\otimes B/tB)^\dag\colon (\m/t\m)^{\dag\oplus n}\to (\m/t\m)$, which is surjective.
Since $\m/t\m$ is indecomposable (Lemma \ref{lemmf}), the Nakayama's lemma indicates that $\mathsf{jac}(\End(\m/t\m))\cdot(\m/t\m)\not= \m/t\m$, where $\mathsf{jac}(\End(\m/t\m))$ is the Jacobson radical of $\End(\m/t\m)$.
Therefore, one of the endomorphism $(\phi_1\otimes B/tB)^\dag,\dots,(\phi_n\otimes B/tB)^\dag$ of $\m/t\m$ must be not contained in $\mathsf{jac}(\End(\m/t\m))$, otherwise $(\phi\otimes B/tB)$ cannot be surjective.
This means that one of the $\phi_1\otimes B/tB,\dots,\phi_n\otimes B/tB$ is an isomorphism.
Say $\phi_i\otimes B/tB$ is an isomorphism.
Then the $B$-homomorphism $\phi_i\colon \m^\dag \to \m$ is also surjective.
Both $\m$ and $\m^\dag$ have constant rank, $\phi_i$ must be an isomorphism.
This shows that $\m\cong \m^\dag$.
\end{proof}




\begin{cor}
Let $(R,\m,k)$ be a complete Cohen-Macaulay local ring of dimension one with a canonical module.
Assume $B\coloneqq\m:\m$ is local and $k$ is infinite.
If $R$ is nearly Gorenstein with multiplicity $e(R)\leq 4$, then either $R$ is almost Gorenstein or $\m\cong \m^\dag$.
\end{cor}

\begin{proof}
Take a minimal reduction $(t)$ of $R$.
We have that $B\otimes_R Q(R)=Q(R)$, this means that $B$ has a constant rank as an $R$-module.
Then the multiplicity $e(\m,B)=\ell_R(B/tB)$ of $B$ as an $R$-module is equal to $4$.
If $B$ is a discrete valuation ring, then the statement follows by Theorem 1.8.
So we may assume that $B/tB$ is not a field.
Deduce that
\[
4=\ell_R(B/tB)\geq \ell_B(B/tB)> \dim_k \mathrm{soc}(B/tB)=r(B).
\]
Now we can apply Theorem \ref{th313} and attain the desired statement.
\end{proof}

\begin{ex}
	Let $R=k[[t^5,t^6,t^7]]$.
	Then $R$ is nearly Gorenstein and has multiplicity $5$, however, neither $R$ is almost Gorenstein nor $\m\cong \m^\dag$.
\end{ex}

\section{Numerical semigroup rings}

In this section, we deal with the numerical semigroup rings $(R,\m)$ having an isomorphism $\m\cong \m^\dag$.
 We begin the section with recalling preliminaries on numerical semigroup.
Let $H\subsetneq \mathbb{N}$ be a numerical semigroup.
The set of {\it pseudo-Frobenius numbers} $\mathsf{PF}(H)$ of $H$ is consisting of integers $a\in \mathbb{N}\setminus H$ such that $a+b\in H$ for any $b\in H\setminus\{0\}$.
Then the maximal element $\mathsf{F}(H)$ of $\mathsf{PF}(H)$ is the Frobenius number of $H$. Set $H'=H\cup\{\mathsf{F}(H)\}$.
Then $H'$ is also a numerical semigroup.
A numerical semigroup of the form $H'=H\cup \{\mathsf{F}(H)\}$ for some symmtric numerical semigroup $H$ (see \cite{Kunz} for the definition of symmtric numerical semigroups) is called {\it UESY-semigroup (unitary extension of a symmetric semigroup)}, which is introduced in \cite{Rosales}.
Note that $\mathsf{F}(H)$ is a minimal generator of $H'=H\cup \{\mathsf{F}(H)\}$.
For a numerical semigoup $H$ and a field $k$, the {\it numerical semigoup ring} of $H$ over $k$ is the subring $k[[\{t^a\mid a\in H\}]]$ of $k[[t]]$, where $t$ is an indeterminate.

\begin{lem}
Let $H$ be a numerical semigroup, $k$ is an infinite field and $(R,\m)$ is the numerical semigroup ring $k[[H]]$. Then the following are equivalent.
\begin{itemize}
\item[(1)] $\m$ is self-dual.
\item[(2)] $H$ is a UESY-semigroup.
\end{itemize}
\end{lem}

\begin{proof}
(1) $\implies$ (2): In the case that $H$ is symmetric, or equivalently $R$ is Gorenstein (see \cite{Kunz}), it follows by Corollary 1.5 (3)$\implies$ (2) that $e(R)\le 2$.
Then there is an odd integer $a$ such that $H=\langle 2, a\rangle$.
It can be checked that the subsemigroup $H'=\langle 2, a+2\rangle$ of $H$ is symmetric, and $H\setminus H'=\{a\}$.
Thus $H$ is UESY.

We may assume that $H$ is not symmetric.
By Theorem \ref{thA}, there is a Gorenstein local subring $(S,\mathfrak{n})$ of $R$ such that $R=\mathfrak{n}:\mathfrak{n}$.
Take a value semigroup $v(S)$ of $S$, where $v$ is the normalized valuation of $k[[t]]$, that is, $v$ takes $t$ to $1\in\mathbb{Z}$.
Then $H=v(R)$, and $v(S)$ is symmetric by the result of Kunz \cite{Kunz}.
Since $R\mathfrak{n}\subset \mathfrak{n}$, $v(R)\setminus v(S)$ is contained in $\mathsf{PF}(v(S))$.
Since $v(S)$ is symmetric, $\mathsf{PF}(v(S))=\{F(v(S))\}$.
Thus one has $v(S)\subseteq H \subseteq v(S)\cup \{F(v(S))\}$.
Therefore, $H$ should be equal to $v(S)\cup \{F(v(S))\}$.
In particular, $H$ is UESY.

(2) $\implies$ (1): Describe $H$ as $H=H'\cup \{\mathsf{F}(H')\}$ with a symmtric numerical semigoup $H'$.
Set $S=k[[H']]$.
Then $\End_S(\m_S)$ is isomorphic to
\begin{align*}
\m\colon \m&=(t^a\mid a\in \mathbb{Z}\text{ such that for any }b\in H'\setminus\{0\},\ a+b\in H' )S\\
&=(t^a\mid a\in H'\cup PF(H'))S
=(t^a\mid a\in H)S=R.
\end{align*}

Thus by our theorem (Theorem \ref{thA}), the maximal ideal $\m$ of $R$ is self-dual.
\end{proof}

\begin{prop}
Let $H=\langle a_1,\dots,a_n\rangle$ be a symmetric numerical semigroup minimally generated by $\{a_i\}$ with $2<a_1<a_2<\cdots<a_n$ and $H'\coloneqq H\cup\{\mathsf{F}(H)\}$.
Put $S=k[[H]]$ over an infinite field $k$ and $R=k[[H']]$.
Then 
the maximal ideal of $R$ is quasi-decomposable.
\end{prop}

\begin{proof}
Denote by $\m_R$ the maximal ideal of $R$.
We will prove that the maximal ideal $\m_R/(t^{a_1})$ of $R/(t^{a_1})$ has a direct summand $I$ generated by the image of $t^{\mathsf{F}(H)}$, and $I\cong k$ as an $R$-module.
Since $t^{\mathsf{F}(H)}$ is a minimal generator of $\m_R$, it is enough to show that $\m_Rt^{\mathsf{F}(H)}\subseteq t^{a_1}R$.
So what we need to show is that $\mathsf{F}(H)+a_i-a_1\in H$ for all $i\not=1$ and $2\mathsf{F}(H)-a_1 \in H$.
These follow by the fact that $\mathsf{F}(H)$ is the largest number in $\mathbb{N}\setminus H$ and the inequalities $a_i-a_i>0$ and $\mathsf{F}(H)-a_1>0$.
\end{proof}

\section{Further characterizations}

The goal of this section is to give characterizations of local rings $R$ such that there exists a one-dimensional local hypersurface $(S,\mathfrak{n})$ such that $R\cong \End_S(\mathfrak{n})$.

\begin{prop}
	Let $(R,\m)$ be a Cohen-Macaulay local ring of dimension one.
	Assume that $R$ has a canonical module and infinite coefficient field $k$.
	Then the followings are equivalent.
\begin{enumerate}[\rm(1)]
	\item
	There is a local hypersurface $(S,\mathfrak{n})$ such that $R\cong \End_S(\mathfrak{n})$.
	\item 
	$e(R)\le 2$, or $R$ has type $2$ and a canonical ideal $I$ such that $I^2=\m I$ and $\ell_R(R/I)=2$.
	\item
	$e(R)\le 2$, or $R$ has embedding dimension $3$, and a canonical ideal $I$ such that $I^2=\m^2$.
\end{enumerate}
\end{prop}

\begin{proof}
	(1)$\implies$(2): Assume $e(R)>2$ and $R$ satisfies (1).
	Then $R$ is not Gorenstein, and $I\coloneqq \mathfrak{n}$ is a canonical ideal of $R$.
	Since $S$ is a hypersurface and not a ，, $\ell_R(I/I^2)=\ell_S(\mathfrak{n}/\mathfrak{n}^2)=2$.
	It forces the equality $I^2=\m I$, since $I$ is not a principal ideal.
	
	(2)$\implies$(1): Consider the case that $e(R)\le 2$.
	Then by the proof of Corollary \ref{co1} (1) $\implies$ (6), there is a Gorenstein local ring $(S,\mathfrak{n})$ such that $R\cong \End_S(\mathfrak{n})$ and $e(S)=\edim S$.
	In particular, $e(S)\le 2$ and $S$ is a hypersurface.
	Now consider the case that $R$ has type 2 and a canonical ideal $I$ such that $I^2=\m I$ and $\ell_R(R/I)=2$.
	One has equalities $\ell_R(I/I^2)=\ell_R(I/\m I)=2$.
	Put $S\coloneqq k+I$.
	Then $S$ is Gorenstein local with a maximal ideal $\mathfrak{n}\coloneqq I$, and $R\cong \End_S(\mathfrak{n})$ (Lemma \ref{lmlmlm}).
	We can compute the embedding dimension $\edim S$ as follows:
\[
\edim S=\ell_S(\mathfrak{n}/\mathfrak{n}^2)=\ell_R(I/I^2)=2.
\]
	Therefore, $S$ should be a hypersurface.
	(2)$\implies$ (3): We may assume $R$ has type $2$.
	By the implication (2)$\implies$(1), we can calculate the embedding dimension of $R$ as $\edim R\le \edim S+1=3$, where $(S,\mathfrak{n})$ is a hypersurface with $R\cong \End_S(\mathfrak{n})$.
	Since $R$ is not Gorenstein, $\edim R$ should be equal to $3$.
	This means $\ell_R(\m/\m^2)=3$.
	On the other hand, one has
	\[
	\ell_R(\m/I^2)=\ell_R(\m/I)+\ell_R(I/I^2)=1+\ell_R(I/\m I)=1+2=3.
	\]
	So the inclusion $I^2\subseteq \m^2$ yields that $I^2=\m^2$.
	The direction (3)$\implies$(2) also follows by similar calculations.
\end{proof}

\begin{ques}
	For a Cohen-Macaulay local ring $(R,\m)$ of dimension one, when is there a local complete intersection $(S,\mathfrak{n})$ with an isomorphism $R\cong \End_S(\mathfrak{n})$?
\end{ques}



\begin{ac}
The author is grateful to his supervisor Ryo Takahashi for giving him kind advice throughout the paper, and to Osamu Iyama for his helpful comments on Theorem \ref{thA}.
The author is also grateful to Luchezar Avramov for useful comments.
Finally, the authors thank the referee for reading the paper carefully and giving fruitful comments and helpful suggestions.
\end{ac}

\end{document}